\documentclass[12pt]{amsart}
\usepackage{geometry}   %????????
\usepackage[colorlinks,citecolor = red, linkcolor=blue,hyperindex]{hyperref}
\usepackage{euscript,eufrak,verbatim, mathrsfs}
\usepackage[psamsfonts]{amssymb}
\usepackage{bbm}
\usepackage{graphicx}
 \usepackage{float}
\usepackage{float, tikz}

 \usepackage[all, cmtip]{xy}

\usepackage{upref, xcolor, dsfont}
\usepackage{amsfonts,amsmath,amstext,amsbsy, amsopn,amsthm}
\usepackage{enumerate}

\usepackage{url}

\usepackage{mathtools}
\usepackage{bookmark}

 \usepackage{euscript}
\usepackage{helvet}         % selects\textbf{\textbf{•}} Helvetica as sans-serif font
\usepackage{courier}        % selects Courier as typewriter font
\usepackage{type1cm}        % activate if the above 3 fonts are
\usepackage{multicol}        % used for the two-column index
\usepackage[bottom]{footmisc}% places footnotes at page bottom

\newtheorem{theorem}{Theorem}[section]
\newtheorem*{theorem*}{Theorem B} 
\newtheorem{lemma}[theorem]{Lemma}

\newtheorem{proposition}[theorem]{Proposition}

\newtheorem*{definition*}{Definition}
\newtheorem*{remark*}{Remark}

\newtheorem*{observation*}{Observation}

\newtheorem*{assumption*}{Assumption}

\theoremstyle{definition}

\theoremstyle{remark}

\newcommand{\C}{\mathbb{C}}

\newcommand{\Var}{\mathrm{Var}}

\newcommand{\Conf}{\mathrm{Conf}}

\begin{document}

\title[DPP with sub-Bergman kernels]{Rigidity of determinantal point processes on the unit disc with sub-Bergman kernels}

\author%[authorlabel1]
{Yanqi Qiu}
\address%[authorlabel1]
{Yanqi QIU: Institute of Mathematics and Hua Loo-Keng Key Laboratory of Mathematics, AMSS, Chinese Academy of Sciences, Beijing 100190, China.}
\email{yanqi.qiu@amss.ac.cn; yanqi.qiu@hotmail.com}

\author
{Kai Wang}
\address{Kai WANG: School of Mathematical Sciences, Fudan University, Shanghai, 200433, China.}
\email{kwang@fudan.edu.cn}

\thanks{Y. Qiu is supported by grants NSFC Y7116335K1,  NSFC 11801547 and NSFC 11688101 of National Natural Science Foundation of China. K. Wang is supported by grants NSFC 11722102.}

\begin{abstract}  
We give natural constructions of number rigid determinantal point processes on the unit disc $\mathbb{D}$ with sub-Bergman kernels of the form
\[
K_\Lambda(z, w) = \sum_{n\in \Lambda}(n+1) z^n \bar{w}^n, \quad z, w \in \mathbb{D},
\]
with $\Lambda$ an infinite subset of the set of non-negative integers.  Our constructions are given both in a deterministic method and a probabilistic method. In the deterministic method, our proofs involve  the classical Bloch functions. 
 \end{abstract}


\maketitle

\section{Introduction}
The present paper is devoted to investigate the existence of some non-trivial and natural {\it number rigid} determinantal point processes over bounded domains in the complex plane.

\subsection{Determinantal point processes}
Let us first recall some notations and concepts of determinantal point processes. Let $M$ be a locally compact and  complete separable metric space equipped with a $\sigma$-finite non-negative measure $\mu$. Denote by $\Conf(M)$ the space of  configurations  over $M$ which consists of  non-negative integer-valued Radon measures on $M$. The topology of vague convergence on the set of Radon measures makes $\Conf(M)$ a Polish space.    We call any Borel probability measure $\mathbb{P}$ on 
 $\Conf(M)$ a   point process on $M$. A point process $\mathbb{P}$ on $M$ is  called determinantal if it admits a reproducing kernel function $K: M\times M \rightarrow \mathbb{C}$  such that 
\begin{eqnarray}
\mathbb{E}_{\mathbb{P}}\Big[\prod_{i=1}^{n}\frac{(\# B_i)!}{(\# B_i-n_i)!}\Big]=\int_{B_1^{n_1}\times \cdots\times B_m^{n_m}} \det\Big[K(x_i,x_j)\Big]_{i,j=1}^n d\mu (x_1)\cdots d\mu (x_n)  
\end{eqnarray} 
for any disjoint bounded Borel sets $B_1,\cdots,B_m$, $ m\geq 1,
 n_i\geq  1,n_1+\cdots+n_m = n.$    Here the counting function $\# B: \Conf(M)\to \mathbb{N} = \{0, 1, 2, \cdots\}$ is defined by $\# B(X)=\int_B d X$ for any $X\in \Conf(M)$. 
 
Determinantal   point processes were introduced by
Macchi in the early seventies to describe random fermion fields in quantum theory \cite{M} and later developed in Soshnikov \cite{S1, S2}, Shairai-Takahashi \cite{ST} and many other authors.  Determinantal   point processes appear in many branches of mathematics, such as  eigenvalues of unitarily invariant random matrices as well as the zeros of random analytic functions on the unit disc \cite{HKPV}.  In most interesting cases, the kernel function $K(x,y)$ yields an integral operator  on $L^2(M, d\mu)$ of locally trace class. In the case that $K$ is a locally trace class operator, a characterization for a point process $\mathbb{P}$ to be determinantal and induced by $K$ is given by the following:  for any pairwise disjoint bounded
Borel sets $B_1,\cdots,B_m$  and any $z_1,\cdots,z_m\in \mathbb{C}$, we have 
\[
\mathbb{E}_\mathbb{P} \Big[\prod_{i=1}^m z_i^{\# B_i } \Big]=\det(Id+\sum_{i=1}^{m}(z_i-1)\chi_{B_i}K\chi_{B_i}).
\]
We refer the reader to   \cite{Bo,HKPV, ST,S1,S2} for further background and details of  determinantal point processes. 

We now recall the definition of the number rigidity of point processes. For a given  Borel subset $C\subset M$, let  ${\mathcal{F}}_C$
  be the $\sigma$-algebra on $\Conf(M)$ generated by all random variables   $\# B$ with all Borel subsets $B\subset C$. For any point process $\mathbb{P}$ on $M$, we denote by $\mathcal{F}_C^{\mathbb{P}}$ the completion of the $\sigma$-algebra ${\mathcal{F}}_C$ with respect to $\mathbb{P}$.  A point process $\mathbb{P}$ on $M$ is called number rigid  if for any bounded Borel set $B\subset M$, the random variable $\# B$ is ${\mathcal{F}}_{M\setminus B}^{\mathbb{P}}$ measurable.  This definition of number rigidity is due to  Ghosh \cite{G} where  he shows that the sine-process is number rigid and Ghosh-Peres \cite{GP} where they show that the Ginibre process and the zero set of Gaussian analytic function on the plane are number rigid. 
  Bufetov\cite{B} shows that determinantal point processes with the Airy, the Bessel and the
  Gamma kernels    are rigid. He indeed  establishes a  general theorem   that rigidity holds for the kernel on the real axis $\mathbb{R}$ with a mild condition of  growth. For more results on the number rigidity of point processes, we refer the reader to \cite{BDQ, BPQ,G-16,  G-JSP, G-survey,  Reda-Naj, OS}.

   However, Holroyd and Soo  showed that  the determinantal point process on the unit
  disc $\mathbb{D}$ with the standard Bergman kernel (with respect to the normalized Lebesgue measure on the unit disc $\mathbb{D}$):
  $$K_{\mathbb{D}}(z,w)=\frac{1}{(1-z \bar{w})^2} = \sum_{n=0}^\infty (n+1) z^n \bar{w}^n $$ 
  is not number rigid  \cite{HS}. See also \cite{BQ} for an alternative proof of this result. More generally,  among many other things, Bufetov, Fan and Qiu \cite{BFQ} showed that for any domain $U$ in the $d$-dimensional complex Euclidean space $\mathbb{C}^d$ without Liouville property (that is, there exists a non-constant bounded holomorphic function  $f: U\rightarrow \C$) and any weight $\omega: U\rightarrow \mathbb{R}^{+}$ locally away from zero, the determinantal point process associated with the reproducing kernel of the weighted Bergman space $L_a^2(U; \omega)$ is not number rigid.  

These negative results lead us to ask  whether there exist natural number rigid determinantal point processes  on a bounded domain of the complex plane (of course, any finite rank orthogonal projection yields a number rigid determinantal point process, so here we are only interested in  infinite rank orthogonal projections).   In this paper, we answer affirmatively  this question with a deterministic and a probabilistic method. It also inspires us to construct a series of examples involving   lacunary series in the Bloch space.

\subsection{Statements of main results}
From now on, we focus on the case of unit disc $\mathbb{D}$ equipped with the normalized Lebesgue measure $dm$. 
We shall consider  determinantal point processes induced by the  orthogonal projection kernels (which we call sub-Bergman kernels) of the form  
\[
K_\Lambda(z,w) = \sum_{n\in \Lambda} (n+1) z^n \bar{w}^n,
\]
where $\Lambda \subset\mathbb{N}= \{0, 1, 2, \cdots\}$ is an infinite subset of $\mathbb{N}$.  Note that $K_\Lambda$ is the orthogonal projection onto the following subspaces of the Bergman space $L_a^2(\mathbb{D}) = L^2(\mathbb{D}) \cap Hol(\mathbb{D})$:
\[
\overline{\mathrm{span}}^{L^2(\mathbb{D})}\Big\{z^n \Big| n \in \Lambda\Big\} \subset L^2_a(\mathbb{D}).
\]

To indicate the idea of our proofs,  in what follows, given any $f=\sum_{n=0}^\infty a_n z^n$, we write   
$$K^f(z,w)=\sum_{n=0}^\infty  {a_n}{(n+1)} z^n \bar{w}^n. $$
In particular, for a subset $\Lambda \subset \mathbb{N}$, we denote 
\begin{align}\label{def-f-lambda}
f_\Lambda(z)=\sum_{n\in \Lambda } z^n.
\end{align}

Recall the definition of  Bloch space on the unit disc
$$\mathcal{B}:=\Big\{f\in Hol(\mathbb{D})\Big|  \| f\|_\mathcal{B}: = \sup_{z\in \mathbb{D}} (1-|z|) |f'(z)|<\infty\Big\}.$$ 

\begin{theorem}\label{main-thm-bloch}
Let $\Lambda \subset \mathbb{N}$ be an infinite subset. Suppose that the function $f_\Lambda$ defined in \eqref{def-f-lambda} satisfies $f_\Lambda \in \mathcal{B}$. Then the determinantal point process on $\mathbb{D}$ induced by the kernel $K_\Lambda(z,w) = K^{f_\Lambda}(z, w) $ is number rigid.
\end{theorem}

We have a  criterion    when  $f_\Lambda$ is included in  the Bloch space involving lacunary series as follows. Let $\Lambda=\{\lambda_1,\lambda_2,\cdots\}$ be a subset  of $\mathbb{N}$
with $\lambda_1<\lambda_2<\cdots$. We say $\Lambda$ lacunary if it satisfies the gap condition
\begin{eqnarray}\label{gapratio}
 \rho_\Lambda:=\liminf_{k\in \mathbb{N} }\frac{\lambda_{k+1}}{\lambda_{k }}> 1. 
\end{eqnarray}
The following characterization   was already  hinted in the proof of   \cite[Lemma 10]{AC} by  Anderson and Shields. We remark a short proof for completeness.
\begin{proposition}\label{prop-lacunary} Let $\Lambda$ be  a subset of $\mathbb{N}$. We have that
	$f_\Lambda \in \mathcal{B}$ if and only if $\Lambda$ is a finite union of some lacunary subsets of  $\mathbb{N}$. 
\end{proposition}

Now we turn to the probabilistic method.  Throughout the paper, suppose that $(\xi_n)_{n=0}^{\infty}$ is a sequence of independent Bernoulli random variables with
\begin{align}\label{def-xi-n}
\xi_n = \left\{
\begin{array}{cl}
1 & \text{with probability $\frac{1}{n+1}$}
\vspace{2mm}
\\
0 & \text{with probability $1 -\frac{1}{n+1}$}
\end{array}
\right..
\end{align}
We shall consider the random analytic function  on the unit disc $\mathbb{D}$:
\begin{align}\label{def-f}
f_\xi=\sum_{n=0}^\infty \xi_n z^n.
\end{align}
By Kolmogorov Three Series Theorem, almost surely, we have $\sum_{n=0}^\infty \xi_n  = \infty$. Therefore, for almost every realization $(\xi_n)_{n=0}^\infty$, the kernel 
\begin{align}\label{def-K-f-xi}
K^{f_\xi}(z, w) = \sum_{n = 0}^\infty \xi_n (n +1) z^n \bar{w}^n
\end{align}
is an	orthogonal projection onto the following infinite dimensional subspace 
\[
\overline{\mathrm{span}}^{L^2(\mathbb{D})}\Big\{z^n \Big| \text{$n\in \mathbb{N}$ such that $\xi_n =1$}\Big\} \subset L^2_a(\mathbb{D})
\]
 and yields a determinantal point process on the unit disc $\mathbb{D}$.  
 
\begin{theorem}\label{main-thm1}
For almost every realization $\xi$, 	the determinantal point process induced by the kernel $K^{f_\xi}(z,w)$ is number rigid.
\end{theorem}

Our probabilistic method yields indeed different construction of number rigid determinantal point processes on $\mathbb{D}$ by the following 
\begin{proposition}\label{prop-nolacunary}
Almost surely, the function $f_\xi$ is not included in the Bloch space $\mathcal{B}$. Or  equivalently, almost surely, the subset 
\[
\Lambda_\xi: = \{n\in \mathbb{N}| \xi_n  =1\}
\]
is not a finite union of lacunary subsets of $\mathbb{N}$. 
\end{proposition}

\section{Rigidity of DPP with sub-Bergman kernels}
 This section is devoted to establish the existence  of  number rigid determinantal point processes on the unit disc $\mathbb{D}$ with sub-Bergman kernels. 

For a bounded measurable compactly supported function $\phi$ on $\mathbb{D}$, we denote by     $S_\phi$   the additive functional on the configuration space $\Conf(\mathbb{D})$  defined by the formula
$$S_\phi(X) = \int_{\mathbb{D}} \phi dX.$$   The following sufficient condition for number rigidity of a point process is showed by Ghosh \cite{G}  and 
Ghosh, Peres   \cite{GP}.

\begin{proposition}[Ghosh and Peres]\label{prop-GP} 
Let $\mathbb{P} $ be a Borel probability measure
	on $\Conf(M)$. Assume that for any $\epsilon > 0$, and any bounded subset $B\subseteq M$, there
	exists a bounded measurable function $\phi: M \rightarrow \mathbb{C}$ of compact support such that $\phi \equiv  1$ on $B$,
	and $\Var S_\phi \leq \epsilon$. Then $\mathbb{P}$ is number rigid. \end{proposition} 

Recall that for a point process $\mathbb{P}$ with an orthogonal projection kernel $K(x,y)$ on the unit disc $\mathbb{D}$ of locally trace class, we have 
\begin{align}\label{def-var-f}
\Var  S_\phi= \frac{1}{2}\int_{\mathbb{D}} \int_{\mathbb{D}}  | \phi(x)-\phi(y)|^2 \cdot |K(x,y)|^2 dm(x) dm(y).
\end{align}

The following lemma is our key  estimation.
\begin{lemma}\label{lem-key} 
For any $\epsilon>0$ and $0<r_0<1$, there exists a bounded measurable function $h: [0, 1) \rightarrow \mathbb{R}$ of compact support on $[0,1)$  such that  $h \equiv  1$ on $[0,r_0]$ and 
	$$\int_{[0,1)}\int_{[0,1)} |h(t)-h(s)|^2 \frac{1}{(1-st)^2 }  ds dt<\epsilon.$$
\end{lemma}
We will postpone the proof for the  lemma  until the next section.

\subsection{Rigid  kernel  via  the Bloch functions}

\begin{proof}[Proof of Theorem \ref{main-thm-bloch}]
Assume that $f_\Lambda \in \mathcal{B}$ and consider the determinantal point process $\mathbb{P}$ on $\mathbb{D}$ induced by the orthogonal projection kernel $K^{f_\Lambda} (z,w)= K_\Lambda(z,w)$.  Since $f_\Lambda \in  \mathcal{B}$, by \cite[Thm 5.13]{Zhu}, there exists $C>0$, such that  
\begin{align}\label{bloch-cond}
(1 - |z|)| f_\Lambda'(z)| \le \| f_\Lambda \|_\mathcal{B} \text{\, and \,} (1-|z|)^2 |f_\Lambda''(z)|\leq C \|f_\Lambda\|_\mathcal{B}. \end{align}
Let $\phi: \mathbb{D} \rightarrow \mathbb{R}_{+}$ be any compactly supported bounded radial function.  Then by \eqref{def-var-f}, we have 
	\begin{eqnarray*}
2 \Var S_\phi   & =&\int_{z\in\mathbb{D}}\int_{w\in\mathbb{D}} |\phi(z)-\phi(w)|^2  |K_\Lambda(z,w)|^2  dm(z) dm(w)\\
&=&\int_{t\in[0,1)}\int_{s\in[0,1)}  |\phi(t)-\phi(s)|^2 \Big[ \int_{\alpha\in[0,2\pi]} \int_{\beta\in[0,2\pi]} |K_\Lambda(te^{i \alpha},se^{i \beta})|^2  \frac{d\alpha}{\pi} \frac{d\beta}{\pi} \Big] ts dt ds. 
\end{eqnarray*}
Note that 
\begin{eqnarray*}
 \int_{\alpha \in[0,2\pi]} |K_\Lambda(te^{i \alpha},se^{i \beta})|^2  \frac{d\alpha}{2\pi} & =&  \int_{\alpha \in [0, 2\pi]} \Big| \sum_{n\in \Lambda} (n +1) t^n  s^n e^{i n (\alpha-\beta)} \Big|^2 \frac{d\alpha}{2\pi} 
\\
&  =  & \sum_{n\in \Lambda} (n+1)^2 (ts)^{2n}. 
\end{eqnarray*}
Moreover,  we have  
\begin{eqnarray*}
\sum_{n\in \Lambda} (n+1)^2 (ts)^{2n} &=& \sum_{n\in \Lambda}\Huge{[} n(n-1) (ts)^{2n-4}t^4s^4+3 n (ts)^{2n-2}t^2s^2+  (ts)^{2n} \Huge{]}\\
&=&  t^4s^4f_\Lambda^{''}(t^2s^2)+3 t^2 s^2 f_\Lambda'(t^2 s^2)+f_\Lambda(t^2 s^2).
\end{eqnarray*}
Therefore, by \eqref{bloch-cond}, there exists $C' > 0$ such that
\begin{eqnarray*}
2 \Var S_\phi &=& 4 \int_{[0,1)}\int_{[0,1)}  |\phi(t)-\phi(s)|^2 \sum_{n\in \Lambda} (n+1)^2 (ts)^{2n}   ts dt ds \\
&= & 4 \int_{[0,1)}\int_{[0,1)}    |\phi(t)-\phi(s)|^2  \Big[t^4s^4f_\Lambda^{''}(t^2s^2)+3 t^2 s^2 f_\Lambda'(t^2 s^2)+f_\Lambda(t^2 s^2)\Big] ts dt ds\\
&\leq & 4C' \int_{[0,1)}\int_{[0,1)}  |\phi(t)-\phi(s)|^2  \frac{1}{(1-t^2s^2)^2} dt ds
\\
& \leq & 4C' \int_{[0,1)}\int_{[0,1)}  |\phi(t)-\phi(s)|^2  \frac{1}{(1-ts)^2} dt ds. 
	\end{eqnarray*}
For any $\epsilon>0$ and any $0< r_0< 1$, if we take $h_{r_0, \epsilon}$ to be the function appeared in Lemma \ref{lem-key} and set 
\[
 \phi_{r_0, \epsilon}(z) = h_{r_0, \epsilon}(|z|),
\]
then $\phi_{r_0, \epsilon}: \mathbb{D}\rightarrow \mathbb{R}$ is a bounded measurable function of compact support  such that $\phi_{r_0, \epsilon} \equiv  1$ on $\{z\in \mathbb{D}: |z|\le r_0\}$
	and 
\[
\Var S_{\phi_{r_0, \epsilon}} \leq \epsilon.
\]
Since any compact subset $B\subset \mathbb{D}$ is included in $\{z\in \mathbb{D}: |z|\le r_0\}$ for some $r_0\in (0, 1)$, by Proposition \ref{prop-GP},   we complete the proof  of Theorem \ref{main-thm-bloch} with the use of Lemma \ref{lem-key}. 
\end{proof}

\begin{proof}[Proof of Proposition \ref{prop-lacunary}]
  Suppose that	$f_\Lambda(z)=\sum_{n\in \Lambda} z^n \in \mathcal{B}$. We have that
  $$\sum_{k\in \Lambda} k   r^k=rf'(r)\leq \frac{\|f\|_{\mathcal{B}}}{1-r}.$$
  For any fixed integer $M \ge 2$, set $r=1-\frac{1}{M} $. Note that there exists $c>0$ such that $r^k>c$ for any $k\leq M$. It follows that for the constant $c'=\frac{1}{c}>0$,
  $$\sum_{1\le k\le M,\, k\in \Lambda}  k \leq  \sum_{1\leq k\leq M, \, k\in \Lambda}  k \frac{r^k}{c}  \leq  c'  \sum_{k\in \Lambda} k r^k \le c' \frac{\| f\|_\mathbb{B}}{1 - r} = c'\| f\|_\mathbb{B}  M  .$$
  This implies that
  $$\sum_{2^n\leq k <2^{n+1}, \, k\in \Lambda } 1\leq \sum_{2^n\leq k <2^{n+1}, \, k\in \Lambda } \frac{k}{2^n}\leq \frac{c' \|f \|_\mathbb{B} 2^{n+1}}{2^n}\leq 2  c' \|f \|_\mathbb{B}.  $$
  To ease the notations, write $q:= [2 c'\|f\|_{\mathbb{B}}]+1$ and $$ \Lambda^{even}:=\Lambda\cap\,\,  \bigcup_{m\in{\mathbb{N}}} I_{2m}, \, \Lambda^{odd}:=\Lambda\cap \bigcup_{m\in{\mathbb{N}}} I_{2m+1},$$  where $I_n=\{k\in {\mathbb{N}}:2^n\leq k <2^{n+1}\}$. Then there exist  subsets $\{\Lambda_i^{even}\}_{i=1}^q$ such that $ \Lambda^{even}  =\cup_{i=1}^{q}{\Lambda_i^{even} } ,$ and 
  each ${\Lambda_i^{even} }$ has at most one element inside $I_{2m}$ and no element included in $I_{2m+1}$ for any $m\in {\mathbb{N}}$. That is, each ${\Lambda_i^{even}}$ is either a finite subset or a subset which satisfies the gap condition \eqref{gapratio}
  	 with  the  gap ratio not less than $2$.  This implies $ \Lambda^{even}$ is the union of at most $q$ many lacunary subsets. The same argument also holds for $\Lambda^{odd}$. Therefore,  $\Lambda$ is the union of at most $2q$ many lacunary subsets.
  
  On the other hand, without loss of generality,  suppose that $\Lambda=\{\lambda_1,\lambda_2,\cdots\}$ is a lacunary subset  with  $\lambda_1< \lambda_2<\cdots$ and $\inf_{k\in \mathbb{N} }\frac{\lambda_{k+1}}{\lambda_{k }}\geq 2$. This implies that
  $$\sum_{2^n\leq k <2^{n+1},\, k\in \Lambda}  1\leq  1 ,\,\, \forall n\in \mathbb{N}$$ and hence for any $r\in (0, 1)$ and any integer $n\ge 1$,  we have 
\[
\sum_{2^n \le k < 2^{n+1}, \, k\in \Lambda} k r^k  \le \sup_{2^n \le k < 2^{n+1}} k r^k \le 2^{n+1} r^{2^n}. 
\]
  Therefore,  by noting 
\[
 \sum_{2^{n-1}\leq k <2^{n}}  4  r^{k} \ge  4 r^{2^n} (2^n - 2^{n-1}) =  2^{n+1} r^{2^n},
\]
for any $r\in (0, 1)$, 
we have 
  $$\sum_{k\in \Lambda,k\geq 2}  k  r^k=\sum_{n=1}^\infty  \sum_{2^n\leq k <2^{n+1},k\in \Lambda}  k  r^k\leq  \sum_{n=1}^\infty    2^{n+1}  r^{2^n}\leq  \sum_{n=1}^\infty  \sum_{2^{n-1}\leq k <2^{n}}  4  r^{k} \leq \frac{4}{1-r}.$$
  This implies that $f_\Lambda\in \mathcal{B}$.
\end{proof}

\subsection{Rigid kernel  via probabilistic methods}

\begin{proof}[Proof of Theorem \ref{main-thm1}]  

By  the definition \eqref{def-K-f-xi} of the kernel $K^{f_\xi}(z, w)$, we have that 
\begin{align}\label{sq-average}
\begin{split}
		 \int_{[0,2\pi]}	 |K^{f_\xi}(te^{i \alpha},se^{i \beta})|^2  \frac{d\alpha}{2\pi}  =&  \int_{[0, 2 \pi]} \Big| \sum_{n=0}^\infty \xi_n (n+1) t^n s^n e^{i n (\alpha - \beta)} \Big|^2 \frac{d\alpha}{2\pi }\\
		=& \sum_{n=0}^\infty \xi_n (n+1)^2 t^{2n} s^{2n}.
\end{split}
\end{align}
For any compact subset $B$ in the unit disc, there exists $0<r_0<1$ such that 
\[
B\subset  \{z\in \mathbb{C}: |z|\le r_0\}.
\]
For such real number $r_0\in (0, 1)$ and any $\epsilon>0$, let $h_{r_0, \epsilon}$ be  the function appeared  in Lemma \ref{lem-key} and set 
\begin{align}\label{def-phi-B}
\phi_{B, \epsilon}(z)=h_{r_0, \epsilon}(|z|).
\end{align}
By \eqref{sq-average}, the definition \eqref{def-xi-n} of the random variables $\xi_n$ and the following elementary identity 
\[
\sum_{n = 0}^\infty (n+1) x^n = \frac{1}{(1 - x)^2},
\]
 we obtain 
	\begin{eqnarray*}
		&&{\mathbb{E}}\left[ \int_{z\in\mathbb{D}}\int_{w\in\mathbb{D}} |\phi_{B, \epsilon} (z)- \phi_{B, \epsilon} (w)|^2  |K^{f_\xi}(z,w)|^2  dm(z) dm(w) \right] \\
			&=&{\mathbb{E}}\left[\int_{t\in[0,1)}\int_{s\in[0,1)}  ts| h_{r_0, \epsilon} (t)- h_{r_0, \epsilon}(s)|^2 dt ds \int_{\alpha\in[0,2\pi]} \int_{\beta\in[0,2\pi]} |K^{f_\xi}(te^{i \alpha},se^{i \beta})|^2  \frac{d\alpha}{\pi} \frac{d\beta}{\pi}\right] \\
		&=& 4 \int_{[0,1)}\int_{[0,1)} ts |h_{r_0, \epsilon}(t)-h_{r_0, \epsilon}(s)|^2 \sum_{n=0}^\infty (n+1)^2 (ts)^{2n}  {\mathbb{E}} \xi_n  dtds \\
			& =  & 4 \int_{[0,1)}\int_{[0,1)}  |h_{r_0, \epsilon}(t)-h_{r_0, \epsilon}(s)|^2 \sum_{n=0}^\infty (n+1)  (ts)^{2n}  ts   dtds \\
		&=& 4 \int_{[0,1)}\int_{[0,1)}    |h_{r_0, \epsilon}(t)-h_{r_0, \epsilon}(s)|^2 \frac{1}{(1-s^2t^2 )^2} ts dt ds \\
&\leq & 4 \int_{[0,1)}\int_{[0,1)}    |h_{r_0, \epsilon}(t)- h_{r_0, \epsilon}(s)|^2 \frac{1}{(1-st )^2}  dt ds 
\\
		& \leq &  4\epsilon.
	\end{eqnarray*}
Now for any integer $n\ge 1$, set 
\begin{align}\label{def-phi-n}
\phi_n(z): = \phi_{B, n^{-2}}(z). 
\end{align}
By the above computation, for each integer $n\ge 1$, we have 
	$${\mathbb{E}}\left[ \int_{z\in\mathbb{D}}\int_{w\in\mathbb{D}} |\phi_n(z)-\phi_n(w)|^2  |K^{f_\xi}(z,w)|^2  dm(z) dm(w) \right] \leq \frac{4}{n^2} $$
and hence 
		$$ \sum_{n=1}^\infty   {\mathbb{E}}\left[\int_{z\in\mathbb{D}}\int_{w\in\mathbb{D}} |\phi_n(z)-\phi_n(w)|^2  |K^{f_\xi}(z,w)|^2  dm(z) dm(w) \right] <\infty. $$
  Levi lemma implies that 
  $$\sum_{n=1}^\infty \int_{z\in\mathbb{D}}\int_{w\in\mathbb{D}} |\phi_n(z)-\phi_n(w)|^2  |K^{f_\xi}(z,w)|^2  dm(z) dm(w) <\infty, \quad a.s.  $$
		It follows that 
		\begin{align}\label{lim-phi-n}
\lim_{n\to\infty}\int_{z\in\mathbb{D}}\int_{w\in\mathbb{D}} |\phi_n(z)-\phi_n(w)|^2  |K^{f_\xi}(z,w)|^2  dm(z) dm(w)  = 0, \quad a.s.   \end{align}
	Note that by \eqref{def-phi-B}, \eqref{def-phi-n} and the property of $h_{r_0, \epsilon}$, we know that for each $n \ge 1$, the function $\phi_n:\mathbb{D}\rightarrow \mathbb{R}$ is a bounded measurable   function of compact support such that $\phi_n \equiv 1$ on $B$. Therefore, by Proposition \ref{prop-GP} and the equality \eqref{def-var-f},  the limit relation  \eqref{lim-phi-n} implies that, for almost every realization of $\xi$,  the determinantal point process induced by the orthogonal projection kernel $K^{f_\xi}(z,w)$ is number rigid. 
\end{proof}

\begin{proof}[Proof of Proposition \ref{prop-nolacunary}]  Let $I_n=(2^{n },2^{n+1} ] \cap \mathbb{N}$ and 
	$\mathcal{N}_n=\sum_{k\in I_n} \xi_k$. We claim that for any integer $C\ge 1$,
	$$\limsup_n \mathcal{N}_n \geq C, \quad a.s. $$

	Note that for  a lacunary  set $\Lambda=\{\lambda_1,\lambda_2,\cdots\}$ with the gap ratio $\rho =\liminf_k \frac{\lambda_{k+1}}{\lambda_{k}}>1$, the subset  $\Lambda\cap  I_n$ has most $[\frac{\log 2}{\log(\rho+1)/2}]+1$ elements when $n$ is  sufficiently large.  More generally,  for an integer set $\Lambda=\cup_{i=1}^p \Lambda_i$ with each lacunary  set $\Lambda_i$ having the gap ratio $\rho_i$,  one has that the set $\Lambda\cap  I_n$ contains at most $\sum_{i=1}^p [\frac{\log 2}{\log(\rho_i+1)/2}]+p$ elements when $n$ is  sufficiently large.   Combining this with the claim,  it follows that almost surely, the subset 
	\[
	\Lambda_\xi = \{k\in \mathbb{N}| \xi_k  =1\}
	\]
	is not a finite union of lacunary sets.
	
	 We next prove the claim. Noting that for $k\in I_n$,
	 $$\mathrm{Prob}[\xi_k=1]=\frac{1}{k+1}\geq \frac{1}{2^{n+1 }+1}  $$
	 and 
	 $$\mathrm{Prob}[\xi_k=0]=1-\frac{1}{k+1}\geq 1-\frac{1}{2^{n } } , $$
	 we have that   for a fixed integer $C\ge 1$,
	\begin{eqnarray*}
	\mathrm{Prob}[\mathcal{N}_n=C]&\geq &\sum_{A\subseteq I_n, |A|=C} \left(\frac{1}{1+2^{n+1 }}\right)^C \left(1-\frac{1}{2^{n } }\right)^{2^{n }-C}\\
	&=& \binom{2^n}{C}\frac{1}{(1+2^{n+1})^C } \left(1-\frac{1}{2^{n } }\right)^{2^{n }-C}.
	\end{eqnarray*}  
Note that
$$\lim_{n\to\infty}\binom{2^n}{C}\frac{1}{(1+2^{n+1})^C }=\lim_{n\to\infty} \frac{2^n(2^n-1)\cdots(2^n-C+1)}{C!(1+2^{n+1})^C}=\frac{1}{ 2^C\,C!} $$
and
$$\lim_{n\to\infty} \left(1-\frac{1}{2^{n } }\right)^{2^{n }-C}=\lim_{n\to\infty} \Huge{\{}[1-\frac{1}{2^{n } }]^{2^n} \Huge{\}}^{ \frac{2^{n }-C}{2^n}}=\frac{1}{e}. $$
Therefore, there exist an integer  $M$ and $\beta>0$ such that for $n>M$,
$$\binom{2^n}{C}\frac{1}{(1+2^{n+1})^C } \left(1-\frac{1}{2^{n } }\right)^{2^{n }-C}>\beta. $$
This implies that, for any integer $n >M$,
	\begin{eqnarray*}
	\mathrm{Prob} [\mathcal{N}_n\geq C]&\geq &   \mathrm{Prob}[\mathcal{N}_n= C]>\beta 
\end{eqnarray*} 
and hence
	\begin{eqnarray*}
	\sum_{n=0}^\infty \mathrm{Prob}[\mathcal{N}_n\geq C]=\infty. 
\end{eqnarray*} 
Noting that the random variables $\mathcal{N}_n$ are   independent,  by Borel-Cantelli lemma,   we have
$$ \limsup_{n\to\infty}\mathcal{N}_n\geq C  \quad a.s. $$
This   completes the proof. 
\end{proof}
	
\section{Proof of Lemma \ref{lem-key}}
In this section  we will find a suitable function $h$ such that the integral in  Lemma \ref{lem-key} is small enough.  Comparing with the trace formula 
$$ \mathrm{tr}(H^*_{\bar g}H_{\bar g})=\int_{\mathbb{D}} \int_{\mathbb{D}} \frac{|g(z)-g(w)|^2}{|1-z\bar{w}|^4} dm(z)dm(w)$$
for $g \in \mathcal{B}$
in  the classical Hankel operator theory in Bergman space \cite{Zhu}, one may guess that a function $h$ on $[0, 1]$ would satisfies the requirement in Lemma \ref{lem-key}   if $z\mapsto g(z) =  h(|z|)$ has a small VMO norm, or the growth of such $g$ is slow than the Poincar\'e metric.

Recall that the   Poincar\'e metric  on $\mathbb{D}$ is defined by  $$\rho(z_1,z_2)=\log \frac{1+\varphi_{z_1}(z_2)}{1-\varphi_{z_1}(z_2)}=\log \frac{|z_1-z_2|+|1-\bar{z}_1 z_2 |}{|z_1-z_2|-|1-\bar{z}_1 z_2 |},$$
where $\varphi_{z_1}(z_2)=\frac{z_2-z_1}{1-\bar{z}_1 z_2}$ is the Mobius transformation on the unit disc $\mathbb{D}$. For $0<r_0<r <1$, we write

\begin{equation}
h(t)= h^{(r_0, r)} (t)= \left\{
\begin{array}{ccl}
1 & & {t \leq r_0}\\
\frac{\rho(t,r )}{\rho(r_0,r )} & & {r_0 \leq t \leq r }\\
0 & & {t \geq r }
\end{array} \right..
\end{equation}
The following theorem implies our  technique lemma in the above section.
 \begin{theorem}
	For fixed $0<r_0<1$, we have that  the integral 
	$$\int_{[0,1)}\int_{[0,1)} |h(t)-h(s)|^2 \frac{1}{(1-st)^2 }  ds dt$$
	tends to $0$ 	when $\varphi_{r_0}(r) \to 1$.
	\end{theorem}

\begin{proof}When $0 < t,s<r_0$ or $r < t,s<1 $, then the expression of the integral  is
	equal to zero. So we shall estimate our integral over the following three domains
	$$ 0<t <r_0<r< s ;\,\,r_0< t < s <r;\,\,r_0< t <r< s . $$
	
	The first case:  the integral $(I)$ over the domain  $ 0<t <r_0<r< s$ can be calculated  explicitly. 
	\begin{eqnarray*}
		(I)  = \int_0^{r_0} dt    \int_{r}^1 ds  \frac{1}{(1-st)^2 } &=& \int_r^{1} ds  \left[   \frac{1 }{s(1-st) }\Big|_{0}^{r_0} \right]
\\
		&=& \int_r^{1}  \frac{r_0 }{ 1-r_0 s}  ds\\
		&=&\log \frac{1-r_0 r}{1-r_0}=\log \frac{1+r_0}{1+r_0  \varphi_{r_0}(r)}.
	\end{eqnarray*}
	Therefore, for fixed $r_0$, the integral $(I)$ tends to $0$ when $\varphi_{r_0}(r) \to 1$.		
	
	The second case:	We estimate   the integral $(II)$ over the domain $r_0 < t < s< r$:
	$$(II)=\int_{r_0}^{r } dt   \int_{t}^r ds  \frac{\rho^2(s,t)}{\rho^2(r_0,r)}\frac{1}{(1-st)^2 }.$$
	By the substitutions  $t\to \varphi_{r_0}(t), s\to \varphi_{r_0}(s)$, we obtain 
	\begin{eqnarray*}
		& &	\int_{r_0}^{r } dt    \int_{t}^r ds  \frac{\rho^2(s,t)}{\rho^2(r_0,r)}\frac{1}{(1-st)^2 }\\ &=& \int_{0}^{\varphi_{r_0}(r) } \frac{1-r_0^2}{(1+r_0 t)^2}dt    \int_{t}^{\varphi_{r_0}(r)} \frac{1-r_0^2}{(1+r_0 s)^2} ds  \frac{\rho^2(s,t)}{\rho^2(0,\varphi_{r_0}(r))}\frac{1}{(1-\frac{s+r_0}{1+r_0 s}\frac{t+r_0}{1+r_0 t})^2 }\\
		&=&   \int_{0}^{\varphi_{r_0}(r) } dt   \int_{t}^{\varphi_{r_0}(r)}   ds  \frac{\rho^2(s,t)}{\rho^2(0,\varphi_{r_0}(r))}\frac{1}{(1-st )^2 }.
	\end{eqnarray*}
	
	Now make the substitution    $s\to \varphi_t(s) $, we get
	\begin{eqnarray*}
		(II )&=&   \int_{0}^{\varphi_{r_0}(r) } dt   \int_{t}^{\varphi_{r_0}(r)}   ds  \frac{\rho^2(s,r)}{\rho^2(0,\varphi_{r_0}(r))}\frac{1}{(1-st )^2 }\\
		&=&   \int_{0}^{\varphi_{r_0}(r) } dt \int_{0}^{\varphi_t(\varphi_{r_0}(r)) } ds \frac{\rho^2(s,0)}{\rho^2(0,\varphi_{r_0}(r))}\frac{1}{ 1- t^2   }\\
		&\leq & \frac{1}{\rho^2(0,\varphi_{r_0}(r))} \int_{0}^{\varphi_{r_0}(r) } \frac{1}{ 1- t^2   } dt \int_{0}^{1 } \rho^2(s,0) ds\\
		&=&  \frac{C_2 }{2 \rho (0,\varphi_{r_0}(r))},
	\end{eqnarray*}
	where $C_2:=\int_{0}^{1 } \rho^2(s,0) ds<\infty$. Therefore,
	integral $(II)$ tends to $0$ when $\varphi_{r_0}(r)\to 1.$
	
	The third case:	We  now estimate   the integral $(III)$
	$$(III)=\int_{r_0}^{r } dt    \int_{r}^1 ds  \frac{\rho^2(r,t)}{\rho^2(r_0,r)}\frac{1}{(1-st)^2 }.$$
	With the substitutions    $t\to \varphi_{r}(t),s\to \varphi_{r}(s) $, the   change of variables formula yields that
	\begin{eqnarray*}
		(III )&=& \int_{r_0}^{r } dt    \int_{r}^1 ds  \frac{\rho^2(r,t)}{\rho^2(r_0,r)}\frac{1}{(1-st)^2 }\\
		&=&   \int_{-\varphi_{r_0}(r)}^{0 } dt \int_{0}^{1 } ds \frac{\rho^2(0,t)}{\rho^2(0,\varphi_{r_0}(r))}\frac{1}{ (1-st)^2  }\\
		&\leq & \frac{1}{\rho^2(0,\varphi_{r_0}(r))} \int_{-1}^{0 }  \rho^2(t,0) dt \int_{0}^{1 } ds\\
		&=&  \frac{C_2}{  \rho^2(0,\varphi_{r_0}(r))}.
	\end{eqnarray*}
	Therefore,
	integral $(III)$ tends to $0$ when $\varphi_{r_0}(r)\to 1,$ which completes the proof. 
\end{proof}

\end{document}